\documentclass[10pt,a4paper]{amsart}

\theoremstyle{plain}
\newtheorem{theorem}{Theorem}[section]
\newtheorem{lemma}[theorem]{Lemma}

\newtheorem{corollary}[theorem]{Corollary}
\theoremstyle{definition}

\theoremstyle{remark}
\newtheorem{remark}[theorem]{Remark}

\def\bin #1#2 {\left( \matrix { #1 \cr #2 \cr } \right) }

\begin{document}

\title[Projective curves not contained in  hypersurfaces of given degree]
{On the genus of  projective curves not contained in hypersurfaces
of given degree}

\author{Vincenzo Di Gennaro }
\address{Universit\`a di Roma \lq\lq Tor Vergata\rq\rq, Dipartimento di Matematica,
Via della Ricerca Scientifica, 00133 Roma, Italy.}
\email{digennar@mat.uniroma2.it}

\abstract Fix integers $r\geq 4$ and $i\geq 2$ (for $r=4$ assume
$i\geq 3$). Assuming that the rational number $s$ defined by the
equation $\binom{i+1}{2}s+(i+1)=\binom{r+i}{i}$ is an integer, we
prove an upper bound for the genus of a reduced and irreducible
complex projective curve in $\mathbb P^r$, of degree $d\gg s$, not
contained in hypersurfaces of degree $\leq i$. It turns out that
this bound coincides with the Castelnuovo's bound for a curve of
degree $d$ in $\mathbb P^{s+1}$. We prove that the bound  is sharp
if and only if there exists an integral surface $S\subset \mathbb
P^r$ of degree $s$, not contained in hypersurfaces of degree $\leq
i$. Such a surface, if existing, is necessarily the isomorphic
projection of a rational normal scroll surface of degree $s$ in
$\mathbb P^{s+1}$. The existence of such a surface $S$ is known for
$i=2$ and $i=3$. It follows that, when $i=2$ or $i=3$, the bound is
sharp, and the extremal curves are isomorphic projection in $\mathbb
P^r$ of  Castelnuovo's curves of degree $d$ in $\mathbb P^{s+1}$.

Quadric and cubic hypersurfaces. Projection of a rational normal
scroll surface. Maximal rank.


\endabstract
\maketitle

\section{Introduction}
The aim of this note is to improve \cite[Proposition 1.3 and
Proposition 1.4]{DGnew}, and to simplify the proof of \cite[Theorem
1.2]{DGnew}. In fact, we are going to prove the following Theorem
\ref{Thmi}. We refer to Remark \ref{Rm1} below for some comments on
the claim.

\begin{theorem}\label{Thmi}
Fix integers $r\geq 4$ and $i\geq 2$ (for $r=4$ assume $i\geq 3$).
Assume that the rational number $s$ defined by the equation
\begin{equation}\label{defofs}
\binom{i+1}{2}s+(i+1)=\binom{r+i}{i}
\end{equation}
is an integer. Let $C\subset \mathbb P^r$ be a reduced and
irreducible complex curve, of degree $d$ and arithmetic genus
$p_a(C)$, not contained in a hypersurface of degree $\leq i$. Assume
$d\gg s$, and divide $d-1=ms+\epsilon$, $0\leq\epsilon\leq s-1$.
Then
\begin{equation}\label{Cast}
p_a(C)\leq \binom{m}{2}s+m\epsilon.
\end{equation}
Moreover, this bound is sharp if and only if there exists an
integral surface in $\mathbb P^r$ of degree $s$, not contained in
hypersurfaces of degree $\leq i$. In this case, every extremal curve
$D$ is not a.C.M., and it is contained in a flag $S\subset T\subset
\mathbb P^r$, where $S$ is a surface of degree $s$, uniquely
determined by $D$, not contained in hypersurfaces of degree $i$, and
$T$ is a hypersurface of $\mathbb P^r$ of degree $i+1$. Furthermore,
$S$ is the isomorphic projection in $\mathbb P^r$ of a rational
normal scroll surface $S'$ of degree $s$ in $\mathbb P^{s+1}$, and,
via this isomorphism, $D$ corresponds to a Castelnuovo's curve of
$S'$ of degree $d$.
\end{theorem}

\begin{remark}\label{Rm1}
$(i)$ For instance, if $i=2$, the number $s$ defined by the equation
(\ref{defofs}) is an integer if and only if $r$ is not divisible by
$3$ \cite[p. 484, line 12 from below ]{BE}. If $i=3$, the number $s$
is an integer if and only if $r=1,2,9,10,11,18,19,27,29$ modulo
$36$. When $r=4$, the number $s$ is an integer if and only if
$i=1,2,5,10$ modulo $12$. When $r=5$, the number $s$ is an integer
if and only if $i=1,2,7,11,13,17,22,23,26,$
$31,37,38,41,43,46,47,53,58$ modulo $60$.

\smallskip
$(ii)$ For the case $r=4$ and $i=2$, we refer to \cite[Theorem
1.1]{DGnew}.

\smallskip
$(iii)$ The bound appearing in (\ref{Cast}) is the Castelnuovo's
bound for a non-degenerate integral curve of degree $d$ in $\mathbb
P^{s+1}$ \cite[p. 87, Theorem (3.7)]{EH}.

\smallskip
$(iv)$ Recall that $h^0(\mathbb P^r,\mathcal O_{\mathbb
P^r}(i))=\binom{r+i}{i}$, and that, if $S'\subseteq \mathbb P^{s+1}$
is a rational normal scroll surface of degree $\sigma$, then
$h^0(S',\mathcal O_{S'}(i))=\binom{i+1}{2}\sigma+(i+1)$. Moreover,
notice that for an integral surface in $\mathbb P^r$ ($r\geq 4$) of
degree $\sigma$, not contained in a hypersurface of degree $\leq i$,
one has $\binom{i+1}{2}\sigma+(i+1)\geq\binom{r+i}{i}$
\cite[(11)]{DGnew}. When the number $s$ defined by (\ref{defofs}) is
not an integer, then one may define $s$ differently as the minimal
integer $\sigma$ such that
$\binom{i+1}{2}\sigma+(i+1)\geq\binom{r+i}{i}$ (compare with
\cite[(11)]{DGnew}). For instance, when $i=2$, we have
$3s+3=\binom{r+2}{2}$ if $r$ is not divisible by $3$, and
$3s+1=\binom{r+2}{2}$ otherwise. More generally, divide
$$
\binom{r+i}{i}-(i+1)=\alpha\binom{i+1}{2}+\beta, \quad 0\leq
\beta\leq \binom{i+1}{2}-1.
$$
Then one has $s=\alpha$ if $\beta=0$, and $s=\alpha+1$ if $\beta>0$.
We hope to give some information in the case $\beta>0$ in a
forthcoming paper.

\smallskip
$(v)$ Let $S$ be an integral surface in $\mathbb P^r$,  not
contained in hypersurfaces of degree $\leq i$, of degree $s$ given
by the equation (\ref{defofs}). In view of Lemma \ref{lemma} (see
below), $S$ is the isomorphic projection in $\mathbb P^r$ of a
rational normal scroll surface $S'\subset \mathbb P^{s+1}$ of degree
$s$. Therefore, since $h^0(\mathbb P^r, \mathcal I_S(i))=0$, the
restriction map $h^0(\mathbb P^r,\mathcal O_{\mathbb P^r}(i))\to
h^0(S,\mathcal O_{S}(i))$ is bijective (compare with $(iv)$ above).
It follows that $S$ is of maximal rank (compare with \cite[p. 482,
line 14 from below]{BE}). We deduce that {\it the existence of an
integral surface in $\mathbb P^r$, not contained in hypersurfaces of
degree $\leq i$, of degree $s$ given by the equation (\ref{defofs}),
is equivalent to the existence of a rational normal scroll surface
of degree $s$ in $\mathbb P^{s+1}$ which can be projected
isomorphically into $\mathbb P^r$ as a surface of maximal rank}. By
\cite[Theorem 2]{BE} one knows that, for $i=2$ and $i=3$, a general
projection in $\mathbb P^r$ of a smooth rational normal scroll
surface of degree $s$, is of maximal rank. Therefore, at least for
$i=2$ and $i=3$, the bound (\ref{Cast}) is sharp. We do not know
whether there are other values of $i$ for which this is true.

\smallskip
$(vi)$ Taking into account the proof of Theorem \ref{Thmi} (see
below), one may explicit the assumption $d\gg s$. In fact,  an
elementary computation, that we omit, proves that it suffices to
assume $d>179$ in the case $i=2$ and $r=5$, and
$$
d>\max\left\{\frac{2s}{r-2}\prod_{j=1}^{r-2}[(r-1)!s]^{\frac{1}{r-1-j}},\,\,
\frac{4s}{r-2}(s+1)^3\right\}
$$
otherwise (compare with \cite[Section $2$, $(v)$, $(vi)$,
$(vii)$]{DGnew}). Notice that, for $r\geq 4$ and $i\geq 2$, one has
$\frac{4s}{r-2}(s+1)^3\geq is$.
\end{remark}

In view of previous Remark \ref{Rm1}, $(v)$, we have:

\begin{corollary}
If $i=2$ or $i=3$, the bound (\ref{Cast}) is sharp.
\end{corollary}

\section{The proof of Theorem \ref{Thmi}}


\maketitle
\begin{lemma}\label{lemma}
Fix integers $r\geq 4$ and $i\geq 2$ (for $r=4$ assume $i\geq 3$).
Assume that the rational number $s$ defined by the equation
(\ref{defofs}) is an integer. Let $S\subset\mathbb P^r$ be an
integral surface of degree $s$,  not contained in a hypersurface of
$\mathbb P^r$ of degree $\leq i$. Then   $S$ is the isomorphic
projection in $\mathbb P^r$ of a rational normal scroll surface
$S'\subset \mathbb P^{s+1}$ of degree $s$.
\end{lemma}

\begin{proof}
It suffices to prove that $h^0(S,\mathcal O_S(1))=s+2$ (notice that
$s\geq 5$).

Let $\Sigma\subset\mathbb P^{r-1}$ be a general hyperplane section
of $S$. For every $j\geq 1$ one has $h^0(\Sigma, \mathcal
O_{\Sigma}(j))\leq 1+js$ \cite[(10)]{DGnew}. Hence, from the natural
exact sequence $0\to \mathcal O_S(j-2)\to \mathcal O_S(j-1)\to
\mathcal O_{\Sigma}(j-1)\to 0$, we obtain $h^0(S,\mathcal
O_S(1))\leq s+2$, and:
$$
h^0(S,\mathcal O_S(i-1))\leq h^0(S,\mathcal O_S(1))+\sum_{j=2}^{i-1}
[1+js]\leq i+\binom{i}{2}s.
$$
We deduce that if $h^0(S,\mathcal O_S(i-1))= i+\binom{i}{2}s$, then
$h^0(S,\mathcal O_S(1))=s+2$. Therefore, it suffices to prove that
$h^0(S,\mathcal O_S(i-1))= i+\binom{i}{2}s$. Since $h^0(\mathbb
P^r,\mathcal I_S(i-1))=0$, from the natural exact sequence $0\to
\mathcal I_S(i-1)\to \mathcal O_{\mathbb P^r}(i-1)\to \mathcal
O_{S}(i-1)\to 0$, we get:
$$
h^1(\mathbb P^r,\mathcal I_S(i-1))=h^0(S,\mathcal
O_S(i-1))-\binom{r+i-1}{i-1}\leq i+\binom{i}{2}s-\binom{r+i-1}{i-1}.
$$
Hence, it suffices to prove that
\begin{equation}\label{itsuff}
h^1(\mathbb P^r,\mathcal
I_S(i-1))=i+\binom{i}{2}s-\binom{r+i-1}{i-1}.
\end{equation}

To this purpose notice that, since $h^0(\mathbb P^r,\mathcal
I_S(i))=0$, from the natural exact sequence $0\to \mathcal
I_S(i-1)\to \mathcal I_{S}(i)\to \mathcal I_{\Sigma,\,\mathbb
P^{r-1}}(i)\to 0$, we have:
\begin{equation}\label{itsuff2}
h^0(\mathbb P^{r-1},\mathcal I_{\Sigma,\,\mathbb P^{r-1}}(i))\leq
h^1(\mathbb P^{r},\mathcal I_{S}(i-1))\leq
i+\binom{i}{2}s-\binom{r+i-1}{i-1}.
\end{equation}
Now, from the natural exact sequence $0\to \mathcal
I_{\Sigma,\,\mathbb P^{r-1}}(i)\to \mathcal O_{\mathbb
P^{r-1}}(i)\to \mathcal O_{\Sigma}(i)\to 0$, we get:
$$
h^1(\mathbb P^{r-1},\mathcal I_{\Sigma,\,\mathbb P^{r-1}}(i))=
h^0(\mathbb P^{r-1},\mathcal I_{\Sigma,\,\mathbb
P^{r-1}}(i))-\binom{r+i-1}{r-1}+h^0(\Sigma,\mathcal O_{\Sigma}(i))
$$
$$
\leq i+\binom{i}{2}s-\binom{r+i-1}{i-1}-\binom{r+i-1}{r-1}+(1+is),
$$
and this number is equal to $0$ in view of the equation
(\ref{defofs}). It follows that:
$$
h^0(\mathbb P^{r-1},\mathcal I_{\Sigma,\,\mathbb
P^{r-1}}(i))=i+\binom{i}{2}s-\binom{r+i-1}{i-1}.
$$
And so from (\ref{itsuff2}) we obtain (\ref{itsuff}).
\end{proof}


\begin{proof}[Proof of Theorem \ref{Thmi}.] First we notice that $C$ cannot be contained in a surface of
degree $<s$, because every such a surface is contained in a
hypersurface of degree $i$ \cite[(11)]{DGnew}. On the other hand,
if $C$ is not contained in a surface of degree $<s+1$, then, since
$d\gg s$, by \cite[loc. cit.]{CCD} we have (compare with
\cite[Section $2$, $(v)$, $(vi)$, $(vii)$]{DGnew}):
$$
p_a(C)\leq \frac{d^2}{2(s+1)}+O(d)<
\frac{d^2}{2s}+O(d)=\binom{m}{2}s+m\epsilon.
$$

Therefore, in order to prove Theorem \ref{Thmi}, we may assume $C$
is contained in a surface $S$ of degree $s$, not contained in a
hypersurface of degree $\leq i$. By previous Lemma \ref{lemma}, we
know that $S$ is an isomorphic projection of a rational normal
scroll surface $S'\subset \mathbb P^{s+1}$. Hence, $C$ is isomorphic
to a curve $C'\subseteq S'$ of degree $d$, in particular
$p_a(C)=p_a(C')$. Since $d\gg s$, by Bezout's theorem $C'$ is
non-degenerate. Therefore, the bound (\ref{Cast}) for $p_a(C)$
follows from Castelnuovo's bound for $p_a(C')$ \cite[loc. cit.]{EH}.
Observe that previous argument shows also that if the bound
(\ref{Cast}) is sharp, then every extremal curve $D\subset \mathbb
P^r$ is contained in a surface $S\subset \mathbb P^r$ of of degree
$s$, not contained in a hypersurface of degree $\leq i$. Conversely,
if such a surface $S$ exists, with notations as before, let
$D'\subset S'$ be a Castelnuovo's curve  of degree $d$. Let
$D\subset S$ be the projection of $D'$ in $S$. Then $D$ is an
extremal curve, and so the bound (\ref{Cast}) is sharp. In fact,
$p_a(D)=p_a(D')$. Moreover, since $d>is$, $D$ cannot be contained in
a hypersurface of degree $i$ by Bezout's theorem (compare with
Remark \ref{Rm1}, $(vi)$). Now, let $D\subset \mathbb P^r$ be an
extremal curve. We just proved that $D$ is contained in a surface
$S$ of degree $s$, a fortiori not contained in a hypersurface of
degree $\leq i$. The surface $S$ is an isomorphic projection of a
rational normal scroll surface $S'\subset \mathbb P^{s+1}$, and, as
before,  via this isomorphism, $D$ corresponds to a Castelnuovo's
curve of $S'$. Since $d\gg s$, again by Bezout's theorem, the
surface $S$ of degree $s$ containing $D$ is unique. By
\cite[(11)]{DGnew} and the definition of $s$, it follows that $S$ is
contained in a hypersurface $T$ of degree $i+1$. Moreover, $D$
cannot be a.C.M.. Otherwise, $H^1(\mathbb P^r,\mathcal I_S(1))=0$
\cite[Section 2, $(vii)$]{DGnew}, and so the restriction map
$H^0(\mathbb P^r,\mathcal I_S(i))\to H^0(\mathbb P^{r-1},\mathcal
I_{\mathbb P^{r-1},\Sigma}(i))$ is onto ($\Sigma\,:=$ a general
hyperplane section of $S$). This is impossible, because, in view of
\cite[(10)]{DGnew}, $\Sigma$ is contained in a hypersurface of
degree $i$, and $S$ is not.
\end{proof}

\end{document}